\documentclass[12pt]{amsart}
\usepackage{a4wide}
\usepackage{latexsym, amsmath, amsfonts, amsthm, amssymb}

\newtheorem{theo}{Theorem}
\newtheorem{prop}[theo]{Proposition}

\newtheorem{rem}[theo]{Remark}
\newtheorem{example}[theo]{Example}

\newtheorem{fact}[theo]{Fact}

\newtheorem{defi}{Definition}

\newcommand{\N}{\ensuremath{\mathbb{N}}} 
\newcommand{\R}{\ensuremath{\mathbb{R}}} 
\newcommand{\C}{\ensuremath{\mathbb{C}}} 

\def\id{{\rm Id}} 
\def\Hess{\mathop{\rm Hess}\nolimits} 
\def\tr{\mathop{\rm tr}\nolimits} 

\def\sp#1#2{\langle#1 \, | \, #2\rangle} 
\def\Sp#1#2{\langle\!\langle#1 \, |\,  #2\rangle\!\rangle} 

\newcommand{\upchi}{\raise1pt\hbox{$\chi$}}

\newcommand{\be}{\begin{equation}}
\newcommand{\ee}{\end{equation}}
\def\benu{\begin{enumerate}}
\def\eenu{\end{enumerate}}

\def\NN{{\rm N}}

\def\beq{\begin{equation}}
\def\eeq{\end{equation}}
\def\beqno{\begin{equation*}}
\def\eeqno{\end{equation*}}
\def\eaeq{\end{aligned}}
\def\baeq{\begin{aligned}}

\def\matd{\mathcal M_d(\C)}
\def\matrd{\mathcal M_d(\R)}
\def\matrn{\mathcal M_n(\R)}

\def\intrn{\int_{\R^n}}

\def\curv{\mathop{\rm{\bf \Theta}}\nolimits}

\def\ad{\mathop{\rm Ad}\nolimits} 

\newcommand{\tnorm}[1]{{\left\vert\kern-0.25ex\left\vert\kern-0.25ex\left\vert #1 
    \right\vert\kern-0.25ex\right\vert\kern-0.25ex\right\vert}}

\title[Matrix valued Pr\'ekopa and Brascamp-Lieb inequalities]{On matrix-valued log-concavity and related Pr\'ekopa and Brascamp-Lieb inequalities}
\author{Dario Cordero-Erausquin}
\address{Institut de Math\'ematiques de Jussieu\\
Sorbonne Universit\'e - UPMC (Paris 6)\\ France.}
\email{dario.cordero@imj-prg.fr}
\begin{document}

\maketitle

\begin{abstract}
We propose a new, self-contained,  approach to H. Raufi's extension of Pr\'ekopa's theorem for matrix-valued log-concave functions. Along the way, new related inequalities are established,  in particular a  Brascamp-Lieb variance inequality for matrix weights. 
\end{abstract}

\tableofcontents

\section{Introduction}

The present note is motivated by Raufi's paper~\cite{R} on matrix-valued log-concavity. One of our goals is to give a proof of  Raufi's extension of Pr\'ekopa's inequality that does \emph{not} use complex variables or complex geometry. 

Pr\'ekopa's theorem~\cite{P} is a fundamental result about convexity and integration that gives a functional form of the Brunn-Minkowski inequality for convex sets. It says that marginals of log-concave functions are log-concave. A function $\alpha:\R^n \to \R^+$ is log-concave if $\log(\alpha)$ is concave, that is if we can write $\alpha = e^{-\varphi}$ with $\varphi$ convex on $\R^n$. Pr\'ekopa's theorem asserts that given a log-concave function $g:\R^{n_0+n_1}\to \R^+$, the function $\alpha:\R^{n_0}\to \R^+$ defined by
$$\alpha(t):= \int_{\R^{n_1}} g(t,y) \, dy $$
is again log-concave. 
 
If we want to extend this result to functions taking their values in the cone  $\matrd^+$ of  positive operators on $\R^d$, or equivalently of $d\times d$ positive symmetric matrices, instead of $\R^+$, that is moving from $d=1$ to $d>1$, we need first to provide a notion of log-concavity for such functions. Let us mention that we can also work exactly in the same way with $\matd^+$, the cone of Hermitian positive operators. 

This is something known in the complex setting (with plurisubharmonicity in place of convexity). A simple situation in dimension $n=1$ appears when doing complex interpolation of families of Hilbert space: a function on the unit disc $g:\mathbb D\to \matd^+$ will be an interpolating family of Hilbert spaces $(\C^d, g)$ if 
$$\curv^g (z):= \partial_{\overline z} (g^{-1} \partial_z g)=0,$$ 
which is an extremal situation for the condition $g(z)\curv^g (z)\le 0$ for the symmetric matrix $g(z) \curv^g(z)$; if $d=1$ and if we write $g(z)= e^{-u(z)}$, then we have  $\curv^g (z)= -\frac14\Delta u(z) $, and so we recover that $u=-\log(g)$ is sub-harmonic.

In complex geometry, similar notions for $n,d>1$ have been investigated in connection with notions of curvature for metrics on vector bundles, see~\cite{D}. This amounts to notions of positivity for an operator $\curv^g$ constructed from the $n^2$ operators $\theta_{j,k}^g:=\partial_{\overline z_j}(g^{-1} \partial_{z_k} g)$. 

In analogy with these notions from complex geometry, Raufi~\cite{R} introduces for a smooth function $g:\R^n \to \matrd^+$ a notion of "Log-concave in the sense of Nakano". We favour the terminology  "$\NN$-log-concave" for simplicity.   We will present in details this notion and the corresponding operator $\curv^g$ in the next section. If $d=1$, it amounts to $\curv^g = \Hess(\log(g))\le 0$ as expected.  Although examples of $\NN$-log-concave functions are for the moment limited in the real world, there is (at least!) a remarkable result.

\begin{theo}[Raufi~\cite{R}]\label{prek}
Let $g:\R^{n_0+n_1}\to \matrd^+$ be a $C^2$ function with values in the symmetric positive $d\times d$ matrices. Assume that for every $t\in \R^{n_0}$, we have $\int_{\R^{n_1}} |g(t,y)|\, dy<+\infty$ where $|\cdot|$ is a norm on $\matrd$ and introduce
$$\alpha(t) := \int_{\R^{n_1}} g(t,y) \, dy \in \matrd^+.$$
If $g$ is \NN-log-concave on $\R^{n_0+n_1}$, then $\alpha$ is \NN-log-concave on $\R^{n_0}$. 
\end{theo}

For $d=1$, it is the result of Pr\'ekopa we discussed above. The assumptions in Raufi's theorem are slightly different than the one we put.

The proof by Raufi is rather sinuous and builds upon several tools from complex analysis and complex geometry. First he complexifies the problem, using $\C^n$ in place of $\R^n$; eventually he will use only functions that depend on the real parts of the variables. He introduces some complex "manifold", which is an (infinite dimensional) hermitian vector bundle (the fibers are weighted Bergman spaces) and reduces the problem to showing that this bundle has a negative Nakano curvature, extending to higher dimensions a result by Berndtsson~\cite{bo}. He then introduces another  bundle, with the help of Fourier transform , for which he needs to establish a vector valued description of Bergman spaces with log-concave weights over tube domains through Fourier transform; this step is probably of independent interest. The new bundle is isometric to the previous one, and finally, for this bundle, he proves the desired curvature estimate using vector valued versions of H\"ormander's $\overline\partial$-estimates. This long way brings  several technical difficulties.   

The fact that some H\"ormander type $L^2$ estimate could be useful was to be expected, starting from the author's curvature computation in~\cite{CE} and from later deep generalizations by Berndtsson (for instance in~\cite{bo}).  But the rest of Raufi's proof, and in particular the use of Fourier transform, is more surprising from the point of view of convex analysis and classical approaches to Brunn-Minkowski type inequalities. 

One of our goals is to provide a different, somehow "classical" and direct approach to the problem and to Theorem~\ref{prek}, inspired by Brascamp-Lieb's approach~\cite{BL1}  to Pr\'ekopa's theorem (see also the survey~\cite{CK}), in particular without using complex variables. Namely, we  we want to:
\begin{enumerate}
\item[i)] Compute the second derivatives $\curv^\alpha$ by "direct" differentiation.
\item[ii)] Find a way to relate this "second derivative in time" (variables $t$)  to "derivatives in space" (variables $y$), using log-concavity. This reduces the problem to variance/spectral type inequality with respect to the density $g(t,y)\, dy$ on $\R^{n_1}$. 
\item[iii)] Establish and prove a general variance/spectral type inequality on $\R^{n_1}$ that applies to i)-ii) for $t\in \R^{n_0}$ fixed. 
\end{enumerate}

In the classical case  case where $d=1$ and we have a log-concave function 
$$g(t,y)=e^{-\varphi(t,y)}$$ 
where $\varphi$ is a convex function on $\R\times \R^n$ (i.e. $n_0=1,\, n_1=n$),
the first step is easy. A straightforward computation  gives, for $\alpha(t) = \int_{\R^n} e^{\varphi(t,y)}\, dy$, that 
\begin{multline}
- \alpha(t) \partial_{tt}^2 \log\alpha (t) = -\partial^2_{tt}\alpha (t) + \frac1{\alpha(t)} (\partial_t \alpha (t))^2  \\
=\int_{\R^n} \partial^2_{tt} \varphi((t,y) \,e^{-\varphi(t,y)} -\int\Big(\partial_t \varphi(t,y) - \mbox{$\frac1{\int_{\R^n}  \,e^{-\varphi(t,\cdot)}} \int_{\R^n} \partial_t \varphi(t,\cdot)
 \,e^{-\varphi(t,\cdot)}$}\Big)^2  \,e^{-\varphi(t,y)}\, dy \label{second1}
 \end{multline}
For the second step,  because  the determinant of the Hessian of $\varphi$, in all the variables $(t,y)$,  is nonnegative, we have that
\beq\label{convexity} \partial^2_{tt} \varphi \ge (\Hess_y \varphi)^{-1} \nabla_y \partial_{t}\varphi\cdot  \nabla_y \partial_{t}\varphi.
\end{equation}
This seems to require that $\Hess_y \varphi>0$, which is actually harmless by approximation. However, there is a way around this difficulty using the associated quadratic form and its polar, as we will show later (this will avoid to discuss approximation or extra assumptions). Property~\eqref{convexity} is very much related to the  homogeneous (real) Monge-Amp\`ere equation; indeed, if instead of $\ge $ we put $=$, this corresponds to a solution of the HRMA equation. Finally,  the third step $iii)$ corresponds, when $d=1$,  to the  Brascamp-Lieb(-H\"ormander) variance inequality, that can be stated as follows: given a convex function $\varphi$ on $\R^n$,
with $\int_{\R^n} e^{-\varphi}<+\infty$, then for every smooth $u\in L^2(e^{-\varphi})$ we have
\be\label{classicalBL}
\int_{\R^n}\Big( u(y) - \frac1{\int_{\R^n} e^{-\varphi}} \int_{\R^n} u e^{-\varphi} \Big)^2 \, e^{-\varphi(y)}\, dy
\le \int_{\R^n}  (\Hess_y \varphi)^{-1} \nabla u (y) \cdot \nabla u(y) \, e^{-\varphi(y)}\, dy
\ee
A direct combination of~\eqref{second1}-\eqref{convexity} and~\eqref{classicalBL} applied at fixed $t$,  to $\varphi(y)= \varphi(t, y)$ and to $u(y) = \partial_t\varphi(t, y)$ gives that $-\partial^2_{tt} \ge  0$, as wanted.

As alluded above, let us mention that both in~\eqref{convexity} and~\eqref{classicalBL}, we can replace $(\Hess_y \varphi )^{-1}(v)$, $v\in \R^n$, by
the  $\displaystyle Q_y^\circ(v) = \sup\big\{v \cdot w \, ;\; Q(w) \le 1\big\} $, the polar of quadratic form $Q_y(w):= (\Hess_y \varphi) w \cdot w$, so that everything is well defined only with the assumption that $\Hess_y \varphi$ is nonnegative.

The matrix situation where $d>1$ brings some complications and requires to establish some new properties of log-concave functions, but the principle of proof works exactly the same.  In the last Section \S4 we will explain in Fact~\ref{comp} how to organise the computation of $\curv^\alpha$ and we will establish in Proposition~\ref{MA} an analogue of~\eqref{convexity} for \NN-log-concave functions. We note that unlike the $d=1$ situation, we cannot assume that $n_0=1$ because \NN-log-concavity does not reduce to a one-dimensional property.

In  Section \S3, we will establish independently a  Brascamp-Lieb variance inequality for matrix weights, which is maybe the main new result of the present paper (and gives some justification for the operator $\curv^g$ used in the definitions below).  It relies on an $L^2$-analysis of the Laplace operator associated with a \NN-log-concave potential. It will then  be used in Section~\S4 to conclude the proof of Theorem~\ref{prek}.

\section{\NN-log-concave matrix valued functions}

We start with some  elementary notation from linear algebra. 

Let $(E, \sp{\cdot}{\cdot}, \|\cdot\|)$ be  a (finite dimensional) real Hilbert space  (later $E=\R^d$ or $E=\R^N \otimes \R^d=\mathcal M_{d,N}(\R)$ the space of $d\times N$ matrices). Given a symmetric positive operator $g$ on $E$, we denote by 
 $\sp{\cdot}{\cdot}_g$ the associated scalar product on $E$, that is 
 $$\forall u,v \in E , \qquad \sp{u}{v}_g  = \sp{g u}{v}, \qquad \|u\|_g := \sqrt{\sp{gu}{u}},$$
If we denote by $\succeq$ (resp. $\succeq_g$) the order on nonnegative operators (resp. associated with the scalar product $g$) on $E$, then we have,  for operator $C$ on $E$, 
 \begin{eqnarray*}
 \Big( C \textrm{ $g$-symmetric with }\  C \succeq_g 0 \Big)&\Longleftrightarrow  &\Big( gC \textrm{ symmetric with }\  gC \succeq 0 \Big)\\
 &\Longleftrightarrow& \sp{gCu}{v}=\sp{gu}{Cv}, \textrm{ and } \sp{gCu}{u}\ge 0, \  \ \forall u,v \in E.
 \end{eqnarray*}
 With such $C$, we can associate the nonnegative quadratic form,
 $$\forall u \in E, \qquad Q_{g, C} (u) := \sp{C u}{u}_g = \sp{g C u}{u}.$$
 The corresponding $g$-\emph{polar} quadratic form is given by
 $$ \forall v\in E, \qquad Q_{g, C} ^\circ (v):= \sup_{Q_{g,C}(u)\le 1} \sp{u}{v}_g^2. $$
 By $2$-homogeneity, the $g$-polar form can also be defined in terms of Legendre's transform, as
 \beq\label{leg}
  \frac12Q_{g, C} ^\circ (v) = \sup_{u\in E} \big\{\sp{u}{v}_g - \frac12Q_{g, C}(u)\big\}. 
  \eeq
When $C$ is moreover invertible, that is when $gC $ is a positive symmetric operator on $E$,  which means that  $Q_{g, C}$ is an Euclidean norm, then $gC^{-1}$ is a positive symmetric operator and the $g$-polar quadratic form $Q_{g, C} ^\circ$ satisfies
 $$ \forall v \in E, \qquad  Q_{g, C} ^\circ (v) = \sp{C^{-1} v}{v}_g = \sp{g C^{-1} v}{v} $$
 or equivalently, $ Q_{g, C} ^\circ = Q_{g, C^{-1}}$.
 This can be seen by diagonalizing $C$ in a $g$-orthonormal basis, for instance.
 
 In the sequel, we will denote by $\sp{\cdot}{\cdot}$ the standard scalar product on $\R^N$ for any $N\in \N^\ast$, and on $\mathcal M_{d , N}(\R)$.

On $\R^n$ and $\R^d$ we use the canonical basis, and in particular we will identify through it operators on $\R^d$ and $d\times d$ matrices. We denote by $\matrd^+$ the positive cone of  positive operators on $\R^d$, that is of $d\times d$ positive symmetric matrices. This will be the fixed range of our densities $g$. 

Let $n,d\ge 1$. Our setting is the following. We are given a  matrix-valued function 
$$g:\R^n \to \matrd^+ \subset \R^{\frac{n(n+1)}2}$$
Log-concavity of $g$ will be defined locally in terms of "second derivatives", so we will assume that $g$ is $C^2$-smooth. We introduce, for $j,k=1, \ldots , n$, and $x\in \R^n$,
\begin{eqnarray*}
\theta^g_{j,k}(x) &:=& \partial_{x_k} ( g^{-1} \partial_{x_j} g)\ \in \matrd \\
& = & g^{-1} \Big[ \partial^2_{x_k x_j} g - (\partial_{x_k} g) g^{-1} (\partial_{x_j} g)\Big] \in \matrd.
\end{eqnarray*}
Derivatives are performed on each entry of the matrix, and so the result remains indeed a matrix.  
Note that for fixed $x\in \R^d$, for $u,v\in \R^d$ and $1\le k,j\le n$, we have
\begin{eqnarray}
 \sp{\theta^g_{j,k} u}{v}_g &=& \sp{ (\partial^2_{x_k x_j} g) u}{v} - \sp{ (\partial_{x_k} g) u}{g^{-1} (\partial_{x_j} g) v} \notag\\
 &=& \sp{ (\partial^2_{x_k x_j} g) u}{v} - \sp{g^{-1} (\partial_{x_j} g) u}{g^{-1} (\partial_{x_k} g) v}_g \label{basicf}.
 \end{eqnarray}
In particular the $g$-adjoint of $\theta^g_{j,k}$ is $\theta^g _{k,j} $: 
\beq\label{sym}
 \sp{\theta^g_{j,k} u}{v}_g =\sp{ u}{\theta^g_{k,j}v}_g 
 \eeq
We have $d^2$ functions $\theta^g_{j,k} :\R^n \to \matrd$. But here and later we will often omit to write the dependence in $x$ so that  $\theta^g_{i,j} $ refers at the same time to a  $\matrd$-valued function and to an element of $\matrd$.  We next collect all these matrices and form 
$$\curv^g = [\theta^g_{j,k}]_{1\le j,k\le n}.$$
Again, we omit the dependence in $x$, so below we are really discussing $\curv^g (x)= [\theta^g_{i,j}(x)]_{i,j\le n}$ for $x\in \R^d$ fixed and omitted, and not the function $x\to \curv^g(x)$.  Note that if $d=1$, then $\curv^g $ is the Hessian $n\times n$ matrix of $\log(g)$; we shall come back to this later. 

There are several possible equivalent ways to see this operator $\curv^g$ for fixed $x\in \R^n$. By construction, we have
$$\curv^g  \in \matrn\otimes\matrd.$$

If we interpret $\curv^g$ as an element of $\mathcal M_n(\matrd)$, then we could ask that for very $Y\in \R^n$, the element $\curv^g Y \cdot Y := \sum_{i,j}^n Y_i Y_j \theta^g_{i,j}\in \matrd$ is a nonnegative operator on $(\R^d, \sp{\cdot}{\cdot}_g)$. This is a rather weak requirement, which corresponds to what is known as Griffith's curvature condition in complex geometry. 

There is a stronger and natural notion, which amounts to work with the following canonical identifications that we shall use in the rest of the paper,
$$\curv^g  \in \matrn\otimes\matrd \simeq L(\R^n \otimes \R^d) \simeq L(\mathcal M_{d,n}(\R)).$$
So we will interpret $\curv^g$ as an operator on $\R^n \otimes \R^d \simeq\mathcal M_{d,n}(\R)$. 

An element $U$ of $\R^n \otimes \R^d \simeq\mathcal M_{d,n}(\R)$ will conveniently be described in term of its columns, which are vectors of $\R^d$, 
$$ U=[u_1, \ldots , u_n], \qquad u_1, \ldots, u_n \in \R^d.$$
The action of the operator  $\curv^g$ is as follows: for $y\in \R^n$ and $u \in \R^d$,
$\curv^g(y\otimes u) = \sum_{j,k=1}^n y_j  e_k\otimes(\theta_{j,k}^g u)$,
where $(e_1, \ldots, e_n)$ is the canonical basis of $\R^n$. Equivalently, if $U=[u_1, \ldots, u_n]
\in \mathcal M_{d, n}(\R)$, then 
\beq\label{curvdef1}
\curv^g U = \Big[\sum_{j=1}^n \theta^g_{j,k} u_j\Big]_{k=1, \ldots , n} \in \mathcal M_{d,n}(\R).
\eeq

We will use the scalar product on $\R^n \otimes \R^d \simeq\mathcal M_{d,n}(\R)$ induced by the scalar product $g$ on $\R^d$ (at $x\in \R^n$ fixed), that is by the action of $g$ on the columns. It is consistent to denote it by $\sp{\cdot}{\cdot}_{\id_n\otimes g}$, if we introduce the positive symmetric operator 
$$(\id_n\otimes g) U = gU = [g u_1, \ldots , g u_n], \quad \textrm{for }\  U=[u_1, \ldots , u_n] \in \R^n \otimes \R^d \simeq\mathcal M_{d,n}(\R).$$
Namely,  for two matrices $U=[u_1, \ldots, u_n],V=[v_1, \ldots, v_n] \in \mathcal M_{d,n}(\R)$ we have 
$$\boxed{\sp{U}{V}_{\id_n\otimes g}=  \sp{(\id_n\otimes g) U}{U}=  \tr((gU)^T V) = \sum_{k=1}^n\sp{g u_k}{v_k} = \sum_{k=1}^n  \sp{u_k}{v_k}_g.}$$
If $d=1$, then for two vectors $u,v\in \R^n= \mathcal M_{1,n}(\R)$ we have $\sp{u}{v}_{\id_n\otimes g} = g(x) \sp{u}{v}$. We emphasise that in general our scalar products incorporate the weight $g(x)$ (we will later integrate in $x$). Denote also
$$\|U\|^2_{\id_n\otimes g} := \sp{U}{U}_{\id_n\otimes g}= \sum_{k=1}^n \sp{gu_k}{u_k}$$
the associated Euclidean norm. 
Since the action of $\id_n\otimes g$ on $\R^n$ is trivial, the notation $\sp{\cdot}{\cdot}_g$ would have been lighter; we favour $\sp{\cdot}{\cdot}_{\id_n\otimes g}$ not only for consistency, but also because it reminds us of the size of the matrices we are working with (later $n$ will vary). 

In view of~\eqref{sym}, we note that $\curv^g$ is a $\sp{\cdot}{\cdot}_{\id_n\otimes g}$-symmetric operator, and that for two matrices $U=[u_1, \ldots, u_n],V=[v_1, \ldots, v_n] \in \mathcal M_{d,n}(\R)$  we have
\beq\label{curvdef2}
\sp{\curv^g U}{V}_{\id_n\otimes g} =\sum_{j,k=1}^n \sp{\theta^g_{j,k} u_j}{v_k}_g = \sum_{j,k=1}^n \sp{ u_j}{\theta^g_{k,j}v_k}_g =\sp{U}{\curv^g V}_{\id_n\otimes g} .
\eeq
For later reference, let us explicitly write this quantity, in view of~\eqref{basicf}, as
\beq\label{basicf2}
\sp{\curv^g U}{V}_{\id_n\otimes g}=\sum_{j,k=1}^n  \Big[ \sp{ (\partial^2_{x_k x_j} g) u_j}{v_k} - \sp{g^{-1} (\partial_{x_j} g) u_j}{g^{-1} (\partial_{x_k} g) v_k}_g \Big]
\eeq
This can also be taken as a definition of the operator $\curv^g$. 

We shall also explain below how to interpret the operator $\curv^g$ on $(\mathcal M_{d,n}(\R), \sp{\cdot}{\cdot}_{\id_n\otimes g})$ as as $dn \times dn$ matrix, acting on $(\R^d)^n$, with the usual scalar product. 

Note that in the classical (but now confusing!) case where $d=1$ and $g\in \R^+$  (we should write $g(x)\in \R^+$ but again $x$ is fixed and omitted), we have that $\sp{a}{b}_g = gab$ for $a,b\in \R$ is therefore trivial (we just are just multiplying by the weight $g(x)>0 $ which is irrelevant as far as positivity is concerned) and that $\curv^g= \Hess (\log(g)) \in L(\R^n)$. So let us emphasize that:
\beq\label{dim1}
\textrm{if $d=1$ and $g=e^{-\varphi}$, then $\curv^g = -\Hess \varphi \in L(\R^n)$, }
\eeq
and for $u\in \R^n = \mathcal M_{1,n}(\R)$, 
$$\sp{\curv^g u}{u}_{\id_n\otimes g} = \sp{\curv^g u}{u}_{g \id_n} = \big[-(\Hess \varphi)u\cdot u \big]\, e^{-\varphi} .$$
The minus sign is a bit unfortunate from the point of view of convex analysis, but it is too well entrenched in complex geometry to change it. So throughout the paper we will have to consider $-\curv^g$ to get nonnegative operators.

The next notion is called "Log-concave in the sense of Nakano" by Raufi~\cite{R}, by analogy with the situation in complex geometry. It is a possible extension to $d>1$ of the situation we just described, and also of similar notions we encounter in complex analysis (for instance when doing complex interpolation of Hilbert spaces). We will favour the terminology  "$\NN$-log-concave" for simplicity.  
\begin{defi}[\NN-log-concavity]
Let $g:\R^n \to \matrd^+$ be a $C^2$ function with values in the symmetric positive matrices. We form at every $x$ the operator $\curv^g$ as before. We say that $g$ is \emph{\NN-log-concave} if at every $x\in \R^n$, $\curv^g$ is a nonpositive symmetric operator on $(\R^n\otimes \R^d, \sp{\cdot}{\cdot}_{\id_n\otimes g})= (M_{d,n}(\R) , \sp{\cdot}{\cdot}_{\id_n\otimes g})$, that is,  if
\be\label{nak}
\sp{\curv^g U}{U}_{\id_n\otimes g} \le 0, \qquad \textrm{for every matrix $U\in \mathcal M_{d,n}(\R)$}.
\ee
\end{defi}

Let us rewrite the previous condition in terms of the canonical structure. If we introduce 
$$\widetilde\curv^g := [\tilde \theta_{j,k}^g]_{j,k\le n}, \quad \textrm{where }\ \tilde \theta_{j,k}^g := g \theta^g_{j,k} = 
\partial^2_{x_k x_j} g - (\partial_{x_k} g) g^{-1} (\partial_{x_j} g)\in \matrd ,$$
that we see as an operator on $(\R^{d})^n= \R^{d n}$, that is as an  $(d n \times dn)$ matrix,  acting on $X=(u_1, \ldots , u_d) \in  (\R^{d})^n$ as
$$\widetilde\curv^g X = \Big( \sum_{i=1}^n \tilde \theta_{i,j}^g  u_j\Big)_{i=1, \ldots , n} \in  (\R^{d})^n,$$
then condition~\eqref{nak} is equivalent to the requirement that
$$-\, \widetilde\curv^g  \succeq 0$$
in the usual sense, as a symmetric $(d n \times dn)$ matrix.

We defined \NN-log-concavity on the whole $\R^n$, but actually if $g$ is defined only on a domain $\Omega\subset \R^n$, that is if $g:\Omega \to \matrd^+$ is a $C^2$ function on $\Omega$, we can say that $g$ is \NN-log-concave on $\Omega$ if $\curv^g(x)$ is a nonpositive operator at every $x\in \Omega$. 

If we impose~\eqref{nak} only on rank one matrices $y\otimes u$, then we arrive to the weaker notion of log-concavity "in the sense of Griffith" that we mentioned above. Note that if $n=1$ (or $d=1$, but we want to work with matrices) these two notions coincide. But if $n,d>1$, they differ (see~\cite{R}), and unlike the case of Griffith log-concavity, it is not sufficient to be \NN-log-concave on any one-dimensional affine line of $\R^n$ to be \NN-log-concave on $\R^n$.

An obvious observation that will be used several times is that given a $\NN$-log-concave function $g:\R^{n_0+n_1}\to \matrd^+$, if we freeze some coordinates $t\in \R^{n_0}$, then the function $y\to g(t,y)$ remains an $\NN$-log-concave function from $\R^{n_1}$ to $\matrd^+$. We stress that although the dimension of the variables might vary, the dimension $d$  will remain fixed throughout the paper. 

Let us also mention that if $P\in \mathcal O_d(\R)$ is a a fixed  isometry (for the usual Euclidean structure) of $\R^d$, then the function 
 $$\tilde g := P^{-1} g P :\R^n \to \matrd^+$$
 will be such that for $u,v\in \R^d$ and $1\le k,j\le n$, at any fixed $x\in \R^n$,
 $$  \sp{\theta^{\tilde g}_{j,k} u}{v}_{\tilde g}= \sp{\theta^g_{j,k} Pu}{Pv}_g.$$ 
So $\tilde g$ is log-concave if and only if $g$ is. 

Raufi discusses in his paper several other properties of $\NN$-log-concave functions, and abstract ways to construct such functions. It is not easy though to give explicit non-trivial examples of \NN-log concave functions for $n,d \ge 2$. 

\begin{example}
Let $n=d=2$. We will define a function $g:\Omega\to \mathcal M_2(\R^2)^+ \subset \R^3$ that is \NN-log-concave in $\Omega\subset \R^2$ a neighbourhood of zero. Let $s\ge 0$ be a paramater to be fixed later. For $x=(x_1, x_2) \in \R^2$ define
$$g(x)=g(x_2,x_2) = {\rm Id}_2 - 
\begin{pmatrix}
 s  x_1^2+x_2^2 & x_1 x_2 \\
 x_1 x_2 & s x_1^2+x_2^2  \\
\end{pmatrix}
= \begin{pmatrix}
 1-s  x_1^2-x_2^2 & -x_1 x_2 \\
 -x_1 x_2 &1 -s x_1^2-x_2^2  \\
\end{pmatrix}.
$$
Note that $g$ has indeed values in $\mathcal M_2(\R)^+$ when $(x_1, x_2)$ is close to zero. For $s=0$, this corresponds to the function $g(x)= {\rm Id}_2 - x^T \, x$; when $d=1$ this is the function $1-x^2$,  which is concave and therefore log-concave 
on $\{|x|<1\}$. This unfortunately does not work when $d>1$, as we will see. 

Computation of $\widetilde\curv^g $ at $x=0$ gives
$$\widetilde\curv^g(0) = \begin{pmatrix}
 -2 s & 0 & 0 & -1 \\
 0 & -2 & -1 & 0 \\
 0 & -1 & -2 & 0 \\
 -1 & 0 & 0 & -2 s 
 \end{pmatrix}
$$
whose spectrum is
$$\{-3, -1, -1 - 2 s, 1 - 2 s\}$$
 which belongs to $(-\infty, 0]$ exactly when $s\ge \frac12$. If we fix $s> \frac12$, then the spectrum of $\widetilde \curv^g(x)$ will remain in $(-\infty, 0)$ for $x$ small enough, by continuity,  and $g$ will be \NN-log-concave. 

 \end{example}

Finally, let us mention that all what we say remains true for functions with values in the Hermitian positive $d\times $d matrices, after obvious adaptions. 

\begin{rem}[Hermitian valued matrices]\label{complex}
It is possible to define, in a complete analogous way, \NN-log-concavity for a function
$$g:\R^n \to \matd^+$$
with values in  the set of  Hermitian positive $d\times d$ matrices. At fixed $x$, $g$ is now a Hermitian product on $\C^d$, denoted again by $\sp{\cdot}{\cdot}_g$ and it induces also an hermitian product $\sp{\cdot}{\cdot}_{\id_n\otimes g}$ on $\mathcal M_{d,n}(\C)$. It is maybe not a good idea to insist too much on the fact that $\R^n\otimes \C^d = \C^n\otimes \C^d \simeq  \mathcal M_{d,n}(\C)$: we have fixed an orthonormal basis on $\R^n$, so complexification should be transparent (as is the complexification of $\id_n$). We can then define $\curv^g$ by the same formulas; it is again a $\sp{\cdot}{\cdot}_{\id_n\otimes g}$ hermitian operator on $\mathcal M_{d,n}(\C)$, and \NN-log-concavity is the requirement that it is a nonpositive one. 

Theorem~\ref{prek} was proved by Raufi in in this setting, actually.

This is obviously more general than the case we considered.
We have chosen to work with real symmetric matrices for notational reasons only (because we prefer $\R^n\otimes \R^d$ to  $\R^n\otimes \C^d$), but all our arguments work of course in the positive hermitian case as well. 
\end{rem}


\section{Brascamp-Lieb variance inequality for matrix-valued \NN-log-concave weights}

Given a function $g:\R^n \to \matrd^+$, we denote by $L^2(g)$ the space of Borel functions $F:\R^n\to \R^d$ for which the quantity
$$\intrn \|F\|_g^2 = \intrn \|F(x)\|_{g(x)}^2 = \int \sp{g(x) F(x) }{F(x)} \, dx $$ 
is finite. It is a Hilbert space over functions from $\R^n$ to $\R^d$ with scalar product $\intrn \sp{F}{H}_g $. If $d=1$ and $g=e^{-\varphi}$, it is the usual weighted  space $L^2(e^{-\varphi})$. Note that if $F\in L^2(g)$ and $\int |g|<+\infty$ (where $|\cdot|$ is a norm on the matrices, for  instance, the operator norm), then the vector 
$$\intrn g F = \intrn g(x) F(x)\, dx \in \R^d$$
 is well defined since $\intrn \|gF \| \le \int \sqrt{|g|}\; \|\sqrt g F\| \le \sqrt{\intrn |g| \intrn \|F\|_g^2} <+\infty$. 

Given a differentiable function $F:\R^n \to \R^d$ we write, at $x\in \R^n$,
$$\nabla_x F=\nabla_x F (x) = [\partial_{x_1} F(x), \ldots, \partial_{x_n}F(x) ] \in \mathcal M_{d,n}(\R).$$

The next result is a generalization of the  Brascamp-Lieb variance inequality~\cite{BL1} (anticipated in the complex setting by H\"ormander~\cite{hormander})  to matrix-valued potentials and vector valued functions. 

\begin{theo}[Brascamp-Lieb variance inequality for matrix \NN-log-concave weights]\label{HBL}
Let $g:\R^n \to \matrd^+$ be a \NN-log-concave function. Assume that $\intrn |g| <+\infty$ and set $\mathcal Z := \int_{\R^n} g\in \matrd^+$.

For any vector valued $C^1$ function $F:\R^n \to \R^d$ belonging to $L^2(g)$ we have
$$\intrn \|F(x)- \mathcal Z^{-1}\mbox{$\intrn$}g F \, \|^2_{g(x) }\, dx \le \intrn Q_{\id_n\otimes g,-\curv^g}^\circ(\nabla_x F)\, dx,$$
where $Q_{\id_n\otimes g,-\curv^g}^\circ$ is, at fixed $x\in \R^n$, the $\sp{\cdot}{\cdot}_{\id_n\otimes g(x)}$-polar of the quadratic form $\mathcal U\to -\sp{\curv^g(x) U}{U}_{\id_n\otimes g(x)} $ on $M_{d,n}(\R)$ associated to $-\curv^g(x)$. If $\curv^g$ is almost everywhere invertible as an operator on $\mathcal M_{d,n}(\R)$, then we can also write
$$ \intrn \|F(x)- \mathcal Z^{-1}\mbox{$\intrn$}g F \, \|^2_{g(x) }\, dx \le \intrn\sp{(-\curv^g(x))^{-1}  \nabla_x  F}{\nabla_x F}_{\id_n\otimes g(x)}\, dx.$$
\end{theo}

When $d=1$, we recover the classical Brascamp-Lieb variance inequality~\eqref{classicalBL}. In some sense, the previous theorem gives some extra justification of the relevance of the operator $-\curv^g$ and to its non-negativity (and therefore to the associated notion of \NN-log-concavity).

\begin{rem}[Hermitian valued weights]
If $g:\R^n \to \matd^+$ is a $C^2$ function with values in the positive hermitian $d\times d$ matrices that is \NN-log-concave
as discussed in Remark~\ref{complex}, then the previous Theorem also applies with the obvious adaptations; that is,  the inequalities will hold for any
$C^1$ function $F:\R^n \to \C^d$ belonging to $L^2(g)$.

Indeed, the proof below is completely identical in this (more general) case. 
\end{rem}

The term on the left in the inequalities of Theorem~\ref{HBL} is a variance, that is the square of the $L^2$ norm of the orthogonal projection onto the orthogonal of constant functions, and we also have, using linearity of integration, 
\begin{eqnarray*}
 \intrn \|F(x)- \mathcal Z^{-1}\mbox{$\int$}g F\,  \|^2_{g(x) }\, dx
 &=& \intrn \sp{g(x) F(x) - g(x) \mathcal Z^{-1} \mbox{$\int$}g F }{F(x) -  \mathcal Z^{-1} \mbox{$\int$}g F }\\
 & = &  \intrn \sp{g F}{F}\, dx - 2 \intrn \sp{g(x) F(x) }{\mathcal Z^{-1}\mbox{$\int$}g F }\, dx \\
 & & \qquad + \intrn  \sp{g(x)\mathcal Z^{-1}\mbox{$\int$}g F }{\mathcal Z^{-1}\mbox{$\int$}g F }\, dx \\
 &= &   \intrn \sp{g F}{F}\, dx - \sp{\mbox{$\int$}g F }{\mathcal Z^{-1}\mbox{$\int$}g F } \\
&=& \int_{\R^n} \|F\|_g^2 \, dx - \Big\|\mathcal Z^{-1}\mbox{$\intrn$}\, g F\, dx\Big\|_{\mathcal Z}^2. 
\end{eqnarray*}

\begin{rem}[$\matrd$-operator form]
It is possible to restate the operator inequality for functions with values on $ \mathcal M_{d}(\R)$. N
Given a smooth function
$$A:\R^n \to \matrd$$
 we denote $\nabla A \in \R^n \otimes \R^d\otimes \R^d \simeq L(\R^d, \mathcal M_{d,n}(\R))$ the operator defined, at $x\in \R^n$,  by
 $$\forall a_0 \in \R^d, \qquad  (\nabla_x A) a_0 :=  \nabla_x \big(A(x) a_0\big) \in \mathcal M_{d,n}(\R).$$
 At fixed $x\in \R^n$, this is an operator between Hilbert spaces that admits an adjoint, that we some abuse of notation we denote simply by $\ad_g(\nabla_x A)\in L(\mathcal M_{d,n}(\R), \R^d)$,
 $$ \forall a_0\in \R^d, \forall U\in \mathcal M_{d , n} (\R), \quad \sp{\nabla_x A a_0}{U}_{\id_n\otimes g(x)} = \sp{a_0}{ \ad_{g(x)}(\nabla_x A) U}_{g(x)}.$$
We will also denote by $\ad_g$ the adjoint with respect to $g$ for operators on $\R^d$, that is, for $A\in \matrd$, 
$$\forall a_0, b_0 \in \R^d, \quad \sp{A a_0}{b_0}_{g(x)}=\sp{a_0}{\ad_{g(x)}(A) b_0}_{g(x)},$$
or equivalently, $\ad_g(A) = g^{-1} A^T g$. 
Then the Brascamp-Lieb inequality is equivalent to the following operator valued inequality with respect to the matrix weight $g$. For any smooth function $A:\R^n \to \matrd $, we have the following inequality between symmetric $d\times d$ matrices
\begin{multline*}
\int_{\R^n} g(x) \ad_{g(x)}\big[A(x) -\mathcal Z^{-1}\mbox{$\intrn$g A} \big]\big(A(x) -\mathcal Z^{-1}\mbox{$\intrn$g A} \big) \, dx \\
\succeq
\int_{\R^n} g(x) \ad_{g(x)}\big[\nabla_x A\big] (-\curv^g(x))^{-1}  \nabla_x A \, dx 
\end{multline*}
We pass from one form to the other by testing this matrix inequality on a fixed vector $a_0$ and considering $F= A a_0 : \R^n \to \R^d$. 

\end{rem}

\begin{proof}
Introduce the differential operator $L$ given for $F :\R^n \to \R^d$ of class $C^2$ by
$$L F := \Delta F - \sum_{k=1}^n (g^{-1}\partial_{x_k} g ) \, \partial_{x_k} F $$
which means, coordinate-wise, for $F=(F_1, \ldots, F_d)$, with $F_\ell : \R^n \to \R$, and $x\in \R^n$, 
$$\forall \ell\le d, \qquad  (LF(x))_\ell = \Delta_x F_\ell (x)  + \sum_{k=1}^n\sum_{r=1}^d \big(g(x)^{-1} \partial_{x_k} g(x)\big)_{\ell, r}  \partial_{x_k} F_r(x) .$$

\begin{fact}\label{ipp}
If $F,G :\R^n\to \R^d$ are $C^2$ and compactly supported, then 
$$\int_{\R^n} \sp{LF(x)}{G(x)}_{g(x)}\, dx = - \int_{\R^n} \sp{\nabla_x F(x) }{\nabla_x G(x)}_{\id_n\otimes g(x)} \, dx =-\sum_{k=1}^n \int_{\R^n} \sp{\partial_{x_k} F}{\partial_{x_k} G}_{g(x)}\, dx.$$
\end{fact}
\begin{proof}[Proof of the Fact]
We have
$$ \int \sp{g(x)\Delta F(x)}{G(x)} = \sum_{k=1}^n  \int \sp{g(x)\partial^2_{x_k x_k} F(x)}{G(x)} \,dx = \sum_{k=1}^n\sum_{\ell,r=1}^d  \int (g(x))_{\ell,r}\partial^2_{x_k x_k} F_r(x) G_\ell(x) \,dx.$$
Integration by parts gives
$$\int_{\R^n} (g(x))_{\ell,r}\partial^2_{x_k x_k} F_r(x) G_\ell(x) \,dx = -\int (\partial_{x_k}g(x))_{\ell,r}\partial_{x_k} F_r(x) G_\ell(x) \,dx - \int_{\R^n}(g(x))_{\ell,r}\partial_{x_k} F_r(x) \partial_{x_k}G_\ell(x) \,dx,$$
and therefore
\begin{eqnarray*} \int_{\R^n} \sp{g(x)\Delta F(x)}{G(x)} &=&  -\sum_{k=1}^n\int_{\R^n} \sp{\partial_{x_k}g \partial_{x_k} F}{G} - \sum_{k=1}^n \int_{\R^d} \sp{g\partial_{x_k} F}{\partial_{x_k} G}\\
& =&-\sum_{k=1}^n\int_{\R^n} \sp{g^{-1}\partial_{x_k}g \partial_{x_k} F}{G}_g \;  - \int_{\R^n}\sp{\nabla F}{\nabla G}_{\id_n\otimes g},
\end{eqnarray*}
as claimed. 
\end{proof}

The next useful observation is a Bochner type integration by parts formula which explicitates the connection between $L$ and the curvature operator $\curv^g$. 
\begin{fact}\label{gamma2}
If $F :\R^n\to \R^d$ is a $C^2$compactly supported function, then 
$$\int_{\R^n} \|LF(x)\|^2_{g(x)}\, dx =  \int_{\R^n} \sp{-\curv^g_x \nabla_x F(x) }{\nabla_x F(x)}_{\id_n\otimes g(x)} \, dx + \int_{\R^n}\sum_{j,k=1}^n \|\partial^2_{x_j x_k}F\|_{g(x)}^2\, dx.$$
\end{fact}
\begin{proof}[Proof of the Fact]
Using the previous Fact we can write
$$\int_{\R^n} \sp{LF(x)}{LF(x)}_{g(x)}\, dx =  -\int_{\R^n} \sp{\nabla_x F(x) }{\nabla_x L F(x)}_{\id_n\otimes g(x)}\, dx .$$
Next, we need to understand the commutation between $\nabla$ and $L$. We have, for $j=1, \ldots , n$
$$\partial_{x_j} L F = \Delta \partial_{x_j} F  + \sum_{k=1}^n\partial_{x_j}(g^{-1} \partial_{x_k} g) \partial_{x_k} F + \sum_{k=1}^n (g^{-1} \partial_{x_k} g) \partial^2_{x_j x_k} F =L\partial_{x_j} F + \sum_{k=1}^n \theta^g_{j,k} \partial_{x_k} F.$$
So we have, at any fixed $x\in \R^n$, 
$$\sp{\nabla_x F }{\nabla_x L F(x)}_{\id_n\otimes g} =  \sum_{j=1}^n \sp{\partial_{x_j} F}{\partial_{x_j} LF}_g = \sum_{j=1}^n  \sp{\partial_{x_j} F}{L\partial_{x_j}F}_g + \sum_{k,j=1}^n  \sp{\partial_{x_j} F}{\theta_{j,k}^g \partial_{x_k} F}_g,$$  
that is
$$\sp{\nabla_x F }{\nabla_x L F(x)}_{\id_n\otimes g} = \sp{\curv^g \nabla F}{\nabla F}_{\id_n\otimes g} +  \sum_{j=1}^n  \sp{\partial_{x_j} F}{L\partial_{x_j}F}_g .$$
After integration we we find, using again the previous Fact, 
\begin{multline*}
\int_{\R^n} \|LF\|^2_{g} = -\int_{\R^n}  \sp{\curv^g \nabla F}{\nabla F}_{\id_n\otimes g}+ \int_{\R^n} \sum_{j=1}^n \sp{\nabla \partial_{x_j}F}{\nabla \partial_{x_j}F}_{\id_n\otimes g} \\
= -\int_{\R^n}  \sp{\curv^g \nabla F}{\nabla F}_{\id_n\otimes g}+ \int_{\R^n} \sum_{j,k=1}^n \sp{\partial^2_{x_j,x_k}F}{\partial^2_{x_j,x_k}F}_g.
\end{multline*}
\end{proof}

The proof of the Brascamp-Lieb can be done by introducing, for given $F$ a function $\Psi$ such that $F= L\Psi$. By Fact~\ref{ipp}, a necessary condition on $F$ is that $\int_{\R^d} g F=0$ and this will be achieved by considering $F- \mathcal Z^{-1} \int gF$. Proving existence and regularity of $\Psi$ requires some work, and for our purposes it is sufficient to establish a classical approximation procedure by "nice" functions. 

\begin{fact}\label{density}
In $L^2(g)$, the space
$$\mathcal H :=\big\{L\Psi \; ; \ \Psi:\R^n \to \R^d \textrm{ of class $C^2$ and compactly supported}\big\}$$
is dense in 
$$\Big\{F \in L^2(g) \; ; \ \int_{\R^n} g F = 0\Big\}.$$
\end{fact}
\begin{proof}
The proof is an adaptation of a classical argument (recalled for instance in~\cite{CFM}).
Let $L^2$ be the standard space of square-integrable functions from $\R^n$ to $\R^d$ equipped the standard Euclidean structure $\sp{}{}$.  The space $L^2(g)$ is isometric to $L^2$ through the isometry
$U : L^2(g) \to L^2$ given by
$$U F = \sqrt{g} F ,$$
since
$$\intrn \|UF\|^2 \, dx \intrn \sp{UF}{UF} \, dx = \intrn  \sp{g F}{F} \, dx = \intrn \sp{F}{F}_g \, dx =\intrn \|F\|_g^2 \, dx .$$
The inverse of $U$ is its adjoint, given by $U^\ast H = \sqrt{g}^{-1} H$. The corresponding linear differential operator on $L^2$,
$$L_0 = U L U^\ast$$
 is of the form, for $H:\R^n\to\R^d$,
$$L_0 H = \Delta H - B\nabla H - c H ,$$
where $B: \R^n \to \mathcal M_d(\R)$ incorporates derivatives of $g$ and thus is of class $C^1$ and $c:\R^n \to \R$ incorporates second derivatives of $g$ and is thus of class $C^0$. Note the presence of $B$, which does not appear when $d=1$. So we really have a (non-diagonal) system of differential operators. However regularity matches classical regularization properties of the elliptic operator $\Delta$ on vector valued functions. That is, if $H$ is a function of $L^2$ that verifies in the sense of distributions $L_0 H=0$, then $H$ is $H^2_{loc}=W^{2,2}_{loc}$. Note that it is useful  that $B$ is of class $C^1$ and therefore preserves distributions in $H^{-1}$, which will then fell in $H^1_{loc}$ by $\Delta^{-1}$, and by repeating the argument, we arrive to $H^2_{loc}$ as a pre-image by   $\Delta$ of an element of $L^2$. 

So, let $F:\R^n \to \R^d$ be a function in $L^2(g)$ that is orthogonal to $\mathcal H$. Our goal is to prove that $F$ is constant. If we denote $H= U F$, we will have $L_0 H = 0$ in the sense of distributions on $\R^n$. By the previous discussion, $H$, and therefore $F$,   will be in $H^2_{loc}=W^{2,2}_{loc}$.

If $\theta:\R^n \to \R^+$ is a smooth compactly supported function we have, setting 
$\nabla \theta \otimes F = [(\partial_1 \theta) F  , \ldots , (\partial_n \theta)  F ]\in \mathcal M_{d,n}(\R)$, 
\begin{multline*}
\|\nabla (\theta F)\|_{\id_n\otimes g}^2 = 
\|\nabla \theta \otimes F + \theta \nabla F \|^2_{\id_n\otimes g}
= \|\nabla \theta \otimes F\|_{\id_n\otimes g}^2 + 2 \theta  \sp{\nabla \theta \otimes F }{\nabla F}_{\id_n\otimes g} +  \|\nabla F \|^2_{\id_n\otimes g}   \theta^2 
\end{multline*}
But by integration by parts, 
$$ \int \|\nabla F \|^2_{\id_n\otimes g}  \theta^2  = \int\sp{\nabla F }{\theta^2 g\nabla F}_{\id_n\otimes g}=- \int \sp{F}{ LF}_g \theta^2 - 2 \int   \theta \sp{\nabla \theta \otimes F }{\nabla F}_{\id_n\otimes g},$$
so that, since $LF=0$ almost everywhere, 
$$\int \|\nabla (\theta F)\|_{\id_n\otimes g}^2   = \int \|\nabla \theta \otimes F\|_{\id_n\otimes g}^2 = \int \|\nabla \theta\|^2 \, \|F\|_g^2.$$
If $\theta$ is smooth compactly supported function that is equal to $1$ in a neighborhood of zero, we set $\theta_k(x)= \theta(x/k)$, then 
$$\int \|\nabla (\theta F)\|_{\id_n\otimes g}^2   =\int \|\nabla \theta_k\|\,  \|F\|_g^2 \to 0$$
and therefore $\|\nabla F\|_g=0$ a.e. and  $F$ is constant.
\end{proof}

We can now do the proof of Theorem~\ref{HBL}.  For our $F:\R^n \to \R^d$, introduce $H:\R^n \to \R^d$ given by
$$H(x) = F(x) - \mathcal Z^{-1} \int_{\R^n} gF,$$
so that $\int_{\R^n} g(x) H(x) \, dx = 0 $ and $\nabla H = \nabla F$. We write
\begin{multline}\label{start}
 \int \|F(x)- \mathcal Z^{-1}\mbox{$\int$}g F \|^2_{g(x) }\, dx \\ = \intrn \|H\|^2_g 
 = 2\intrn \sp{H}{L\Psi}_g - \intrn \|L\Psi\|_g^2 +\intrn\|H-L\Psi\|_g^2 
 \end{multline}
for any given $\Psi$ of class $C^2$ and compactly supported. We will concentrate on the two first term of the last expression, which would have been the only one to appear if we could have chosen directly $L\Psi = H$; we have an extra-term, which can be chosen arbitrary small. 

Using Fact~\ref{ipp} and Fact~\ref{gamma2} for $\Psi$, we can write
\begin{multline*}
2\intrn \sp{H}{L\Psi}_g - \intrn \|L\Psi\|_g^2 \\
= -2 \intrn\sp{\nabla_x H }{\nabla_X \Psi}_{\id_n\otimes g(x)} \, dx 
- \int_{\R^n} \sp{\curv^g_x \nabla_x \Psi(x) }{\nabla_x \Psi(x)}_{\id_n\otimes g(x)} \, dx - \int_{\R^n}\sum_{j,k=1}^n \|\partial^2_{x_j x_k}\Psi\|_{g(x)}^2\, dx. \\
\le  -2 \intrn\sp{\nabla_x H }{\nabla_X \Psi}_{g(x)}\, dx 
- \int_{\R^n} \sp{\curv^g_x \nabla_x \Psi(x) }{\nabla_x \Psi(x)}_{\id_n\otimes g(x)} \, dx  \\
\le \intrn Q_{\id_n\otimes g(x), \curv^g(x)}^\circ (\nabla_x H) \, dx =  \intrn Q_{\id_n\otimes g(x), \curv^g(x)}^\circ  (\nabla_x F) \, dx.
.
\end{multline*}
where we used the characterization of the $g(x)$-polar form in terms of Legendre's transform~\eqref{leg}. So from~\eqref{start} we get, for any function $F$ and any $C^2$ compactly supported function $\Psi$,
$$ \int \|F(x)- \mathcal Z^{-1}\mbox{$\int$}g F \|^2_{g(x) }\, dx \le 
 \intrn Q_{\id_n\otimes g(x), \curv^g(x)}^\circ (\nabla_x F) \, dx
 +\intrn\|H  -L\Psi\|_g^2 $$
Taking the infimum over $\Psi$, we conclude thanks to Fact~\ref{density}. 
\end{proof}

\section{Proof of the matrix valued Pr\'ekopa's inequality}

Besides the matrix valued Brascamp-Lieb inequality that we have established above, the proof of Pr\'ekopa's inequality relies on the following property that clarifies the r\^ole of  \NN-log-concavity.

\begin{prop}\label{MA}
Let $g:\R^{n_0+n_1}\to \matrd^+$ be a $C^2$ function with values in the symmetric positive matrices and assume that $g$ is \NN-log-concave. At any given $x\in \R^{n_0+n_1}$, divide the operator $\curv^g$ in blocks: 
$$\curv^g_{0,0} :=[\theta^g_{j,k}(x)]_{1\le j,k\le n_0} \in L(\R^{n_0}\otimes \R^d)\simeq  L(\mathcal M_{d,n_0}(\R)) $$
and 
$$\curv^g_{1,1} :=[\theta^g_{i,j}(x)]_{n_0< j,k\le n_0+n_1}\in L(\R^{n_1}\otimes \R^d)\simeq  L(\mathcal M_{d,n_1}(\R)),$$
which are defined according to~\eqref{curvdef1}-\eqref{curvdef2},
together with the mixed-derivatives operator 
$$\curv^g_{0,1}:=\Big[ \theta^g_{j,k}(x)\Big]_{1\le j\le n_0, n_0<k\le n_0+n_1} \in L(\mathcal M_{d,n_0}, \mathcal M_{d,n_1}) $$
whose action on $V_0 =[v_1, \ldots, v_{n_0}]\in \mathcal M_{d,n_0}$  is,  as expected, given by 
 $$\curv^g_{0,1} V_0 := \Big[\sum_{j=1}^{n_0} \theta^g_{j,n_0+k}v_j \Big]_{k=1, \ldots, n_1}\in \mathcal M_{d, n_1} $$ 

Then, for any given $V_0 \in \mathcal M_{n_0, d}$ and $x\in \R^{n_0+n_1}$, if we denote by $Q^\circ_{1,1}$ the $\sp{\cdot}{\cdot}_{\id_{n_1}\otimes g}$-polar of the quadratic form $W\to -\sp{\curv^g_{1,1}W}{W}_{\id_{n_1}\otimes g}$ associated to $-\curv^g_{1,1}$, we have
$$ \sp{-\curv^g_{0,0} V_0}{V_0}_{\id_{n_0}\otimes g}  \ge Q^\circ_{1,1} (\curv^g_{0,1} V_0  ) .
$$
In case $\curv^g_{1,1}$ is invertible at $x$, this is equivalent to 
$$ \sp{-\curv^g_{0,0} V_0}{V_0}_{\id_{n_0}\otimes g}  \ge \sp{(-\curv^g_{1,1})^{-1}\curv^g_{0,1} V_0}{\curv^g_{0,1} V_0}_{\id_{n_1}\otimes g}.$$
\end{prop} 

As discussed in the introduction, this extension of~\eqref{convexity} is very much related to the homogeneous Monge-Amp\`ere equation. 

\begin{proof}
For fixed $V_0 \in  \mathcal M_{n_0, d}$ and $x\in \R^{n_0+n_1}$, we will test the positivity of the operator $-\curv^g$ on the matrix $[V_0, W]\in \mathcal  \mathcal M_{n_0+n_1, d}$ for an arbitrary matrix $W\in  \mathcal M_{n_1, d}$. The optimal choice would be $W=(-\curv^g_{1,1})^{-1}\curv^g_{0,1} V_0$, but we want to proceed without assuming the invertibility of $\curv^g_{1,1}$. 

Using the notation of the Proposition in the formula~\eqref{curvdef2} and the symmetry property~\eqref{sym}, we have
$$\sp{\curv^g([V_0, W])}{[V_0,W]}_{\id_{n_0+n_1}\otimes g} = \sp{ \curv^g_{0,0} V_0}{V_0}_{\id_{n_0}\otimes g} 
+2\sp{\curv^g_{0,1} V_0}{W}_{\id_{n_1}\otimes g}  + \sp{\curv^g_{1,1} W}{W}_{\id_{n_1}\otimes g}.$$
Since $\sp{\curv^g([V_0, W])}{[V_0,W]}_{\id_{n_0+n_1}\otimes g} \le 0$, by strong log-concavity, we have
$$\sp{ -\curv^g_{0,0} V_0}{V_0}_{\id_{n_0}\otimes g}\ge 2\Sp{\curv^g_{0,1} V_0}{W}_{\id_{n_1}\otimes g} - \sp{-\curv^g_{1,1} W}{W}_{\id_{n_1}\otimes g} .$$
Taking the supremum over $W$ gives the desired inequality. 
\end{proof}

We have now all the ingredients to prove the Pr\'ekopa-Raufi inequality. With the notation of Theorem~\ref{prek}, we have to estimate $\sp{\curv^\alpha V_0}{V_0}_{\id_{n_0}\otimes\alpha}$ at some fixed $x\in \R^{n_0}$ for $V_0=[u_1, \ldots , u_{n_0}]\in\mathcal M_{d,n}(\R)$. 

\begin{fact}[Computing $\curv^\alpha$]\label{comp}
\label{second} With the notation of Theorem~\ref{prek}, 
at any $x\in \R^{n_0+n_1}$  and fixed $V_0=[v_1, \ldots , v_{n_0}]\in\mathcal M_{d,n_0}(\R)$ we introduce
\beq\label{defF}
F:= \sum_{j=1}^{n_0} (g^{-1} \partial_{x_j} g ) v_j \in \R^d.
\eeq
Then, at a fixed $t\in \R^{n_0}$, we have, with the notation $\curv^g_{0,0} $ taken from Proposition~\ref{MA}, 
\begin{eqnarray*}
\lefteqn{\sp{\curv^\alpha V_0}{V_0}_{\id_{n_0}\otimes\alpha}}\\
&= &\int_{\R^{n_1}} \sp{\curv^g_{0,0} V_0}{V_0}_{\id_{n_0}\otimes g(t,y)} \, dy  \\
& & \qquad +
\int_{\R^{n_1}} \sp{F}{F}_{g(t,y)}\, dy - \sp{\alpha^{-1}\mbox{$\int_{\R^{n_1}}$}
gF \, dy}{\alpha^{-1} \mbox{$\int_{\R^{n_1}}$}
gF \, dy }_{\alpha(x)} \\
&= &\int_{\R^{n_1}} \sp{\curv^g_{0,0}(t,y) V_0}{V_0}_{\id_{n_0}\otimes g(t,y)} \, dy +
\int_{\R^{n_1}} \Big\| F(t,y) - \alpha(y)^{-1} \mbox{$\int_{\R^{n_1}}$}
g(t,z) F(x,z)\, dz   \Big\|_{g(t,y)}^2\, dy .
\end{eqnarray*}
\end{fact}
We see that this computation is the generalization of~\eqref{second1}. 
\begin{proof}
We have for $j,k=1, \ldots, n_0$, 
\begin{eqnarray*}
\sp{\theta_{j,k}^\alpha(t) v_j}{v_k}_{\alpha(t)} &=& \sp{\partial^2_{j,k} \alpha(t) v_j}{v_k} - \sp{\partial_j \alpha(t) v_j}{\alpha(t)^{-1}\partial_k \alpha(t) v_k}\\
&=& \int_{\R^{n_1}}  \sp{\partial^2_{j,k} g(t,y) v_j}{v_k} \, dy
-  \sp{\int_{\R^{n_1}} \partial_j g(t,y) v_j \, dy}{\alpha(t)^{-1} \int_{\R^{n_1}} \partial_k g(t,y) v_k\, dy }\\
&=& \int _{\R^{n_1}} \sp{\theta_{j,k}^g (t,y) v_j }{v_k}_{g(t,y)} \, dy \\
& & \quad  + \int_{\R^{n_1}}  \sp{g(t,y)^{-1} \partial_j g(t,y) v_j}{g(t,y)^{-1} \partial_k g(t,y)v_k}_{g(t,y)} \, dy  \\
&  & \qquad -  \sp{\int_{\R^{n_1}} \partial_j g(t,y) v_j \, dy}{\alpha(t)^{-1} \int_{\R^{n_1}} \partial_k g(t,y) v_k\, dy }
\end{eqnarray*}
Summing over all $j,k\in\{1, \ldots, n_0\}$ we find
$$\sp{\curv^\alpha V_0}{V_0}_{\id_{n_0}\otimes \alpha} = \int_{\R^{n_1}} \sp{\curv^g_{0,0} V_0}{V_0}_{\id_{n_0}\otimes g(t,y)} \, dy
+\int_{\R^{n_1}} \sp{F}{F}_g \, dy - \sp{\int_{\R^{n_1}}g F \, dy }{\alpha^{-1}\int_{\R^{n_1}} g F \, dy}$$
For the equality between the expressions, see the computations after Theorem~\ref{HBL}.
\end{proof}

We now combine the previous results. The strong log-concavity, in the form given by Proposition~\ref{MA} combines with Fact~\ref{second} and implies that, at our fixed $t\in \R^{n_0}$ and $V_0=[v_1, \ldots ,v_{n_0}]\in \mathcal M_{d,{n_0}}(\R)$ we have
\begin{multline*}
\sp{-\curv^\alpha V_0}{V_0}_{\id_{n_0}\otimes \alpha}   \ge \\
 \int_{\R^{n_1}} Q_{1,1}^\circ (\curv^g_{0,1} V_0) \, dy 
  -
\int_{\R^{n_1}} \Big\| F(t,y) - \alpha^{-1}(t) \mbox{$\int_{\R^{n_1}}$}
g(t,z) F(t,z)\, dz   \Big\|_{g(t,y)}^2\, dy ,
\end{multline*}
where $Q^\circ_{1,1}$ at $(t,y)$ is the $\sp{\cdot}{\cdot}_{\id_{n_1}\otimes g(t,y)}$-polar of the quadratic form $W\to -\sp{\curv^g_{1,1}W}{W}_{\id_{n_1}\otimes g}$ associated to $-\curv^g_{1,1}$ at $(t,y)$. 
The function 
$$y \to F(t,y)$$
that we denote by $F$ for convenience, given by~\eqref{defF} at $x=(t,y)$ with $t$ fixed,  is such that
$$\nabla_y F=\Big[ \sum_{j=1}^{n_0} \partial_{y_1}\big((g^{-1} \partial_{x_j} g )\big) v_j, \ldots 
\sum_{j=1}^{n_0} \partial_{y_{n_1}}\big((g^{-1} \partial_{x_j} g )\big) v_j\Big]\in \mathcal M_{d,n_1}(\R)$$
and so, with the notation of Proposition~\ref{MA}, we find 
$$\nabla_y F =  \curv^g_{0,1}V_0.$$
Writing the previous inequality in terms of $F$ we arrive to
$$
\sp{-\curv^\alpha V_0}{V_0}_{\id_{n_0}\otimes \alpha} \ge   \\
 \int_{\R^{n_1}} \sp{(-\curv^g_{1,1})^{-1}\nabla_y F}{\nabla_y F}_{\id_{n_1}\otimes g} \, dy 
  -
\int_{\R^{n_1}} \Big\|  F - \alpha^{-1} \mbox{$\int_{\R^{n_1}}$}
 g F   \Big\|_{g}^2\, dy .
$$
In some sense the previous inequality is the local form Pr\'ekopa's inequality. The conclusion now follows from the Brascamp-Lieb inequality of Theorem~\ref{HBL} on $\R^{n_1}$ applied to the log-concave weight  $y\to \tilde g(y) := g(t,y)\in \matrd^+$ (for which $\curv^{\tilde g}$ at $y$ is at $\curv^g_{1,1}$ at $(t,y)$) and to the function $y\to F(t,y)$.

\medskip

\end{document}